\documentclass[a4paper,11pt]{amsart}
\usepackage{latexsym,bm}
\usepackage{mathrsfs,amsmath,amssymb,amsthm,enumerate}
\usepackage[arrow,matrix]{xy}
\usepackage{color}
\usepackage{tikz}

\theoremstyle{plain}
\newtheorem{theorem}{Theorem}[section]
\newtheorem{proposition}[theorem]{Proposition}
\newtheorem{lemma}[theorem]{Lemma}
\newtheorem{corollary}[theorem]{Corollary}
\numberwithin{equation}{section}

\theoremstyle{definition}
\newtheorem{definition}[theorem]{Definition}
\newtheorem{remark}[theorem]{Remark}
\newtheorem{example}[theorem]{Example}
\newtheorem{question}[theorem]{Question}
\newtheorem{convention}[theorem]{Convention}

\newcommand{\C}{\mathbb{C}}

\newcommand{\R}{\mathbb{R}}
\newcommand{\Z}{\mathbb{Z}}

\newcommand{\CP}{\mathbb{C}P}

\newcommand{\balpha}{\boldsymbol{\alpha}}

\DeclareMathOperator{\wed}{wed}

\DeclareMathOperator{\relint}{relint}

\DeclareMathOperator{\conv}{conv}

\begin{document}
\title[Projective toric manifolds]{Wedge operations and a new family of projective toric manifolds}

\author[S.Choi]{Suyoung Choi}
\address{Department of Mathematics, Ajou University, San 5, Woncheondong, Yeongtonggu, Suwon 443-749, Korea}
\email{schoi@ajou.ac.kr}

\author[H.Park]{Hanchul Park}
\address{School of Mathematics, Korea Instutute for Advanced Study (KIAS),
85 Hoegiro Dongdaemun-gu, Seoul 130-722, Republic of Korea}
\email{hpark@kias.re.kr}

\thanks{This research was supported by Basic Science Research Program through the National Research Foundation of Korea(NRF) funded by the Ministry of Science, ICT \& Future Planning(NRF-2012R1A1A2044990).}

\date{\today}

%14M25   	Toric varieties, Newton polyhedra
%57S25   	Groups acting on specific manifolds
%52B11     $n$-dimensional polytopes
%52B35   	Gale and other diagrams
%14J10   	Families, moduli, classification: algebraic theory

%
%\subjclass[2010]{primary 57N65; secondary 57S17, 05E45}
%\keywords{real toric manifold, small cover}

\begin{abstract}

    Let $P_m(J)$ denote a simplicial complex obtainable from consecutive wedge operations from an $m$-gon. In this paper, we completely classify toric manifolds over $P_m(J)$ and prove that all of them are projective. As a consequence, we provide an infinite family of projective toric manifolds.

\end{abstract}
\maketitle

\tableofcontents
\section{Introduction}

A \emph{toric variety} of complex dimension $n$ is a normal complex algebraic variety of dimension $n$ equipped with an effective algebraic action of $(\C^\ast)^n$ with an open dense orbit. Every toric variety corresponds one-to-one to a combinatorial object called a \emph{fan} and the fan is a key object to classify toric varieties. Among toric varieties, we are mainly interested in complete smooth toric varieties which are also called \emph{toric manifolds}.
The \emph{underlying simplicial complex} of a simplicial fan is the face complex of the fan. A toric manifold whose underlying simplicial complex is $K$ is also called a toric manifold over $K$.

There is a classical operation of simplicial complexes called \emph{wedge operation}. For a simplicial complex $K$, the wedge of $K$ at a vertex $v$ of $K$ is denoted by $\wed_vK$ and a simplicial complex obtained by a series of wedges from $K$ can be written as $K(J)$ for a tuple $J$ of positive integers. See \cite{BBCG10} for details. As the authors have shown in \cite{CP13} and \cite{CP15}, the wedge operation plays an important role in classification of toric manifolds. Roughly speaking, for every toric manifold $M$ over a wedge of $K$, there are two toric manifolds $M_1$ and $M_2$ over $K$ called {projections of $M$ over $K$} and $M$ is determined by $M_1$ and $M_2$. Colloquially, if we know all toric manifolds over $K$, then we know all toric manifolds over $K(J)$. In fact, this method (let us call it \emph{classification-by-wedge}) works for a lot of categories of toric spaces -- quasitoric manifolds, topological toric manifolds, omnioriented quasitoric manifolds, almost complex quasitoric manifolds, small covers, and so on.

Let us present some examples of toric manifolds over wedges. The Bott manifolds are a classical example of a family of toric manifolds and they can be also understood as toric manifolds over the boundary complex of the hypercube $I^n = [0,1]^n$. A result of the first author, Masuda, and Suh \cite{CMS10} shows that generalized Bott manifolds are exactly toric manifolds over the boundary complex of a product of simplices.  Note that they are obtained by a sequence of wedges of the boundary complex of $I^n$.

Another remarkable example is toric manifolds of Picard number 3. Their projectivity was shown by Kleinschmidt and Sturmfels \cite{KS91} and the classification of them was completed by {Batyrev} \cite{Ba}. Actually, it can be shown that they are toric manifolds over the wedges of the boundary complex of $I^3$ or the pentagon, and both the classification and the projectivity can be also obtained by the method using wedges by the authors \cite{CP13}, which is far more concise and systematic.

Therefore, a very natural next step in the classification of toric manifolds would be the classification of toric manifolds over $P_m(J)$, where $P_m$ denotes the boundary complex of an $m$-gon. Note that the classification of toric manifolds over polygons -- or toric manifolds of complex dimension 2 -- is well known. In Section~\ref{sec:2}, we accomplish the classification in the language of ``diagrams'' in \cite{CP15}. Furthermore, in Section~\ref{sec:3}, we show the following:
\begin{theorem}\label{thm:cpisproj}
    Every toric manifold over $P_m(J)$ is projective for any $m\ge 3$ and an $m$-tuple $J \in \Z_+^m$.
\end{theorem}
Together with the fact that every generalized Bott manifold is projective, Theorem~\ref{thm:cpisproj} generalizes the result of \cite{KS91}.
Moreover, Theorem~\ref{thm:cpisproj} is a nontrivial fact even though every two $\C$-dimensional toric manifold is projective, because {classification-by-wedge seems to fail for the category of projective toric manifolds or toric Fano manifolds.} In non-smooth case, there exists a non-projective toric orbifold over $\wed_v K$ whose two projections over $K$ are projective; see Example~7.1 of \cite{CP13}.

\begin{question}\label{que:projectiveprojection}
    For a star-shaped simplicial complex $K$ and its vertex $v$, let $M$ be a toric manifold over $\wed_vK$ whose two projections $M_1$ and $M_2$ over $K$ are projective. Then is $M$ projective?
\end{question}

There is a criterion to determine whether a toric manifold is projective or not. It is
a version of Gale duality called \emph{Shephard diagrams} \cite{She71}. Even though the answer of Question~\ref{que:projectiveprojection} is currently unknown, the Shephard diagram works well together with classification-by-wedge since a Shephard diagram of a toric manifold $M$ over a wedge of $K$ is given by Shephard diagrams of projections of $M$ over $K$. This is a key tool to prove Theorem~\ref{thm:cpisproj}.

We call a polytopal simplicial complex $K$ is \emph{weakly combinatorially Delzant} or \emph{WCD} if there is a projective toric manifold over $K$. Theorem~\ref{thm:cpisproj} tells us that for all $m$ and $J$, $P_m(J)$ is WCD and does not support non-projective toric manifolds. We say that such simplicial complexes are \emph{strongly combinatorially Delzant} or \emph{SCD}. Examples of SCD complexes contains joins of simplices (corresponding to generalized Bott manifolds) and $P_m(J)$. We know that wedges or joins of WCD complexes are also WCD since they support canonical extensions introduced in \cite{Ewa86} or products of toric manifolds. Hence, the following question looks natural:

\begin{question}
    Let $K$ and $L$ be SCD complexes and $v$ a vertex of $K$. Then is $\wed_vK$ or $K\ast L$ SCD? Here, $K\ast L$ denotes the simplicial join of $K$ and $L$.
\end{question}

Note that the stellar subdivision does not necessarily preserve SCD property (see Oda's example of non-projective toric manifold in \cite{oda88}), but it does preserve WCD property since it corresponds to the equivariant blow-up of a toric manifold.

The WCD and SCD properties expand hierarchy of star-shaped simplicial complexes. See Figure~\ref{fig:hierarchy}. Let $K$ be a star-shaped simplicial sphere of dimension $n-1$ equipped with an orientation $o$ as a simplicial manifold. Then the characteristic map $(K,\lambda)$ is said to be \emph{positively oriented} or simply \emph{positive} if the sign of $\det(\lambda(i_1),\dotsc,\lambda(i_n))$ coincides with $o(\sigma)$ for any oriented maximal simplex $\sigma=(i_1, \ldots, i_n)\in K$. Note that every fan-giving characteristic map can be positively oriented and therefore the inclusion in Figure~\ref{fig:hierarchy}.

\begin{figure}\label{fig:hierarchy}
\begin{center}
\begin{tikzpicture}[scale=0.11]
    \draw[thick] (-6,-3) rectangle (98,66);
    \draw (45,63) node {star-shaped simplicial complexes};
    \draw[rounded corners=20pt] (0,0) rectangle (90,60);
    \draw (45,57) node {{support $\Z_2$-ch. maps}};
    \draw[rounded corners=15pt] (6,3) rectangle (84,54);
    \draw (45,51) node {{support $\Z$-ch. maps}};
    \draw[rounded corners=10pt] (10,6) rectangle (63,48);
    \draw (36,45) node {support positive ch. maps};
    \draw[rounded corners=10pt] (14,9) rectangle (54,42);
    \draw (33.5,39) node {supp. non-singular fans};
    \draw[rounded corners=10pt] (18,12) rectangle (95,36);
    \draw (73,33) node {polytopal};
    \draw[rounded corners=10pt] (22,15) rectangle (50,33);
    \draw (36,30) node {WCD};
    \draw[rounded corners=5pt] (26,18) rectangle (47,27);
    \draw (36,23) node {SCD};
\end{tikzpicture}
    \caption{Hierarchy of star-shaped simplicial complexes}
\end{center}
\end{figure}
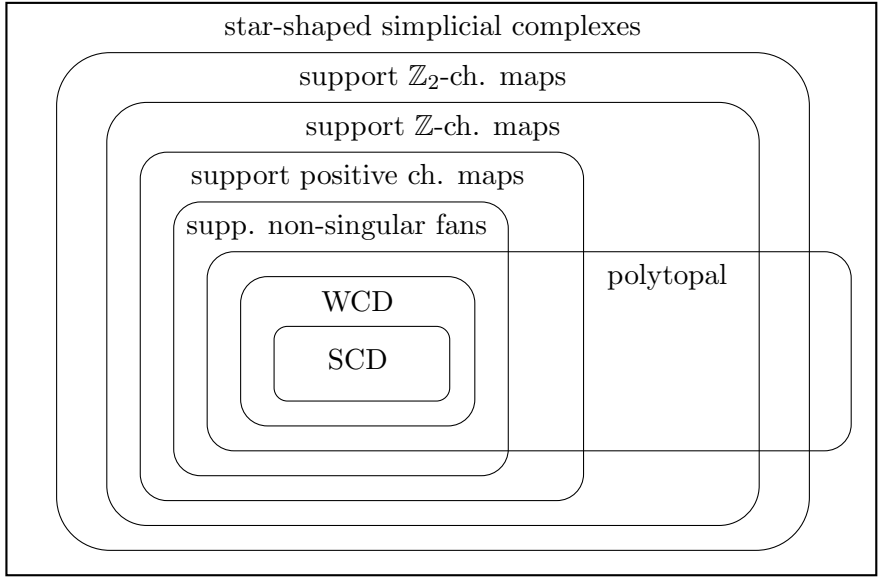

This paper is organized as follows. In Section~\ref{sec:1}, we review the classification of toric manifolds of complex dimension $2$ and give a complete classification of them using the language of fans. In Section~\ref{sec:2}, we classify the toric manifolds over $P_m(J)$ by studying the structure of the diagram $D(P_m)$ as you can see in \cite{CP15}. In Section~\ref{sec:3}, we prove Theorem~\ref{thm:cpisproj}.
\bigskip

\section{Plane fans and toric manifolds of $\C$-dimension 2} \label{sec:1}
The purpose of this section is to give a classification of complete non-singular fans on the plane $\R^2$. Actually, the classification of toric manifolds of complex dimension two was accomplished by Orlik-Raymond \cite{OR}, but we think it is certainly meaningful to prove it again using the language of fans.

Let $r$ be a ray in $\R^n$ starting at the origin. One says that a nonzero vector $v$ \emph{generates} $r$ if $v\in r$. In other words, $r = \{ av \mid a\ge 0\}$.
Let $\Sigma$ be a complete non-singular fan in $\R^2$ with $m$ rays $(m\ge 3)$. Label the primitive vectors generating the rays with
\begin{equation}\label{eqn:planefan}
    v_1,v_2,\dotsc,v_m
\end{equation}
in counterclockwise order. Write $v_i = {x_i \choose y_i} $ and then $\Sigma$ corresponds to the following characteristic matrix
    \[
        \lambda = \begin{pmatrix}
            x_1 & x_2    &   \cdots & x_m \\
            y_1 & y_2     &  \cdots & y_m
        \end{pmatrix}.
    \]
{Throughout this paper, when a fan in $\R^2$ is given, we assume that its rays are labeled in counterclockwise order. We frequently assume that $v_1=\binom10$ and $v_2 = \binom01$ by basis change of $\Z^2$.}

By non-singularity condition, we have the identity $x_iy_{i+1} - x_{i+1}y_i = 1$ for $i=0,\dotsc,m$, where $v_0 = v_m$. By subtracting each hand side of $x_iy_{i+1} - x_{i+1}y_i = 1$ from that of $x_{i-1}y_i - x_iy_{i-1} = 1$, we obtain $y_i(x_{i-1}+x_{i+1}) = x_i(y_{i-1}+y_{i+1})$. Therefore there exists an integer $a_i$ such that $x_{i-1}+x_{i+1} = a_ix_i$ and $y_{i-1}+y_{i+1} = a_iy_i$, so that $v_{i-1}+v_{i+1} = a_iv_i$ for $i=1,\dotsc,m$. This trick using $a_i$ appears in \cite{Ma99}.

For any sequence of vectors $v_0, v_1,\dotsc,v_m=v_0$ corresponding to a complete non-singular fan, the sequence
\[
    v_0, v_1,\dotsc, v_i, v_i+v_{i+1}, v_{i+1}, \dotsc, v_m
\]
also corresponds to a complete non-singular fan. We call this new fan a \emph{blow-up} of the original fan. This {naming} is justified by the fact that this operation corresponds to an equivariant blow-up of the toric manifold.

\begin{theorem}
    If $m\ge 5$, then $\Sigma$ is a blow-up.
\end{theorem}

\begin{proof}
    We have the identity $v_{i-1}+v_{i+1} = a_iv_i$ for $i=1,\dotsc,m$, where $v_0=v_m$ and $v_{m+1}=v_1$. The fan $\Sigma$ is a blow-up if and only if $a_i = 1$ for some $i$. Among $v_1,\dotsc,v_m$, a vector $v_i$ called \emph{locally maximal} if $\|v_i\|\ge \|v_{i-1}\|$ and $\|v_i\|\ge \|v_{i+1}\|$. There always exists a locally maximal vector (for example just pick a vector of maximal length). We can assume that $v_1$ is locally maximal and then $v_m+v_2 = a_1v_1$ and
    \[
        |a_1|(\|v_m\|+\|v_2\|) \le 2 \|a_1v_1\| = 2\|v_m+v_2\| < 2(\|v_m\|+\|v_2\|),
    \]
    concluding that $|a_1| \le 1$. If $a_1 = 1$, it is done. Otherwise, we change the basis so that $v_1 = {1\choose 0}$ and $v_2 = {0\choose 1}$. Then $v_m = {a_1 \choose {-1}}$ ($a_1=0$ or $-1$). Be careful that $v_1$ no longer needs to be locally maximal after changing basis.

    Now, we claim that there is a locally maximal vector other than $v_1,v_2$, and $v_m$. Find a vector $v_i$ of maximal length. If $3 \le i \le m-1$, it is done. If not, notice that $\|v_2\|=1$ and $\|v_m\| = 1$ or $\sqrt 2$ and these are the smallest possible lengths of integral vectors. Between $v_2$ and $v_m$, there are 1 slot for integral vectors of length 1 and 2 (or less) slots for length $\sqrt 2$. Since $m\ge 5$, it is easy to observe that the claim holds.

    Suppose that $v_i$ is the vector in the previous claim. By the same argument as above, $a_i =1,0,$ or $-1$. When $O$ is the origin, observe that
    \[
        a_i = \left\{
                \begin{array}{ll}
                  1, & \hbox{if $\angle v_{i-1}Ov_{i} + \angle v_{i}O v_{i+1} < \pi$;} \\
                  0, & \hbox{if $ \angle v_{i-1}Ov_{i} + \angle v_{i}O v_{i+1} = \pi$;} \\
                  -1, & \hbox{if $\angle v_{i-1}Ov_{i} + \angle v_{i}O v_{i+1} > \pi$.}
                \end{array}
              \right.
    \]
    But in our setting, $\angle v_{i-1}Ov_{i} + \angle v_{i}O v_{i+1} < \pi$. Thus $a_i = 1$ and the proof is complete.
\end{proof}

The following corollary reproves a result of \cite{OR}.

\begin{corollary}
    Every toric manifold of complex dimension two is either $\CP^2$ or obtained by equivariant blow-ups from a Hirzebruch surface.
\end{corollary}

\section{The Classification} \label{sec:2}
There is a fundamental operation on simplicial complexes called \emph{simplicial wedging}. Note that every complete non-singular fan in $\R^2$ has face structure of a polygon. In this section, our objective is to classify complete non-singular fans over a simplicial wedge of an $m$-gon when $m\ge 4$. In the case $m=3$, note that any simplicial wedge of a triangle is the boundary complex of a simplex and the corresponding toric manifold is nothing but a complex projective space. A main tool in the classification is in \cite{CP15}.
Let us review very briefly the concept of diagrams and puzzles. We recommend to see \cite{CP15} for details. Let $K$ be a star-shaped simplicial complex with $m$ vertices and $J = (j_1,\dotsc,j_m)\in \Z_+^m$ a positive integer tuple. Roughly speaking, the \emph{pre-diagram} of $K$, written as $D'(K)$, is an edge-colored graph whose vertices are non-singular fans over $K$ (up to equivalence) and each edge indicates a non-singular fan (up to equivalence) over  $K(J)$. Furthermore, one has a set of subsquares of $D'(K)$ determined only by $K$ called \emph{realizable squares}. The \emph{diagram} of $K$, denoted by $D(K)$, is $D'(K)$ equipped with the set of realizable squares of $D'(K)$.
The diagram $D'(K)$ gives a complete classification of toric manifolds over $K(J)$ for any $J\in \Z_+^m$ in the following sense. Let $G(J)$ be the 1-skeleton of the product of simplices $\prod_{i=1}^m \Delta^{j_i-1}$ suitably edge-colored. Then the following holds.
\begin{theorem}\cite[Theorem~5.4]{CP15}
    Up to D-J equivalence, every toric manifold over $K(J)$ corresponds one-to-one to a graph homomorphism $p\colon G(J) \to D'(K)$ such that
    \begin{enumerate}
        \item $p$ preserves edge coloring,
        \item and every subsquare of $p$ is realizable in $D'(K)$.
    \end{enumerate}
\end{theorem}
Such $p$ is called a \emph{realizable puzzle}.
For later use, let us introduce one more term. A realizable puzzle $p$ is called \emph{irreducible} if for any edge $e\in G(J)$ whose endpoints are $\balpha$ and $\balpha'$, their images $p(\balpha)$ and $p(\balpha')$ are different in $D'(K)$. Otherwise, $p$ is called \emph{reducible}. Reducible realizable puzzles correspond to canonical extensions.

Let $P_m$ be a polygon with vertex set $[m]=\{1,\dotsc,m\}$. To classify toric manifolds over $P_m(J)$, we consider the diagram $D(P_m)$ in the category of toric manifolds.
\begin{proposition}\label{prop:shift} Let $\Sigma_1$ and $\Sigma_2$ be two complete non-singular fans with $m$ rays in $\R^2$ and the matrix
    \[
        \lambda = \begin{pmatrix}
            1 & 0   & x_3   & x_4 & \cdots & x_m \\
            0 & 1   & y_3   & y_4 &  \cdots & y_m
        \end{pmatrix}
    \]
    a characteristic matrix for $\Sigma_1$. Suppose that in the diagram $D(P_m)$, the two fans are connected by an edge colored as 1. Then either of the following holds:
    \begin{enumerate}
        \item two fans are the same.
        \item two fans share a ray generated by $\binom{-1}0$.
    \end{enumerate}
    In the second case, a characteristic matrix for $\Sigma_2$ can be written as the following:
    \[
        \lambda_e := \begin{pmatrix}
            1 & 0 & x_3 & \cdots & x_{\ell-1} & -1 & x_{\ell+1}+e  & x_{\ell+2} - y_{\ell+2}e & \cdots & x_m - y_m e \\
            0 & 1 & y_3 & \cdots & y_{\ell-1} & 0  & y_{\ell+1}=-1 & y_{\ell+2}               & \cdots & y_m
        \end{pmatrix}
    \]
    where $x_\ell = -1,\,y_\ell=0$ for $3\le \ell \le m-1$ and $e \in \Z$.
    Conversely, for every $e\in\Z$, $\lambda_e$ is connected to $\lambda$ by an edge of $D(P_m)$ colored as $1$.
\end{proposition}

\begin{proof}
    Let us consider a characteristic map over $\wed_1 P_m$ corresponding to the edge. Due to (3.1) of \cite{CP15}, we can think of its corresponding standard form:
    \[
    \begin{pmatrix}
        1_1 &   1_2 &   2   & 3   &  4  & \cdots & m   \\ \hline
        1   &   0   &   0   & x_3 & x_4 & \cdots & x_m \\
        0   &   0   &   1   & y_3 & y_4 & \cdots & y_m \\
        -1  &   1   &   e_2 & e_3 & e_4 & \cdots & e_m
    \end{pmatrix},
    \]
    where $e_i$, $2\le i \le m$, is an integer. By adding $-e_2$ times the second row to the third row, one can assume that $e_2 = 0$. The projection with respect to $1_2$ is just $\Sigma_1$ and the projection with respect to $1_1$ has the characteristic matrix
    \[
        \begin{pmatrix}
            1 & 0   & x_3+e_3 & \cdots & x_m+e_m \\
            0 & 1   & y_3   &   \cdots & y_m
        \end{pmatrix}.
    \]
    By the non-singularity condition, one obtains the identity $e_iy_{i+1} - e_{i+1}y_i = 0$ for $3 \le i \le m-1$. Moreover, $x_3 = -1$ and therefore $e_3 = 0$. If $e_i$ is zero for all $i$, the two fans coincide. Otherwise, $y_\ell = 0$ for some $\ell \ge 3$. {Such $\ell$ is unique since there are at most two rays parallel to the $x$-axis.} The characteristic matrix has the form
    \[
        \begin{pmatrix}
            1 & 0 & x_3 & x_4 & \cdots & x_{\ell-1} & x_\ell = -1 & x_{\ell+1}+e_{\ell+1}  & \cdots & x_m + e_m \\
            0 & 1 & y_3 & y_4 & \cdots & y_{\ell-1} & y_\ell = 0  & y_{\ell+1}             & \cdots & y_m
        \end{pmatrix}.
    \]
    From the identity $e_iy_{i+1} - e_{i+1}y_i = 0$, we have
    \[
        e_i = e_{i-1}\cdot\frac{y_{i}}{y_{i-1}} = e_{i-2}\cdot\frac{y_{i-1}}{y_{i-2}}\cdot\frac{y_{i}}{y_{i-1}} = \cdots = e_{\ell+1}\cdot \frac{y_i}{y_{\ell+1}} = -y_i e_{\ell+1}
    \]
    for $i=\ell+1, \ell+2, \dotsc, m$.
\end{proof}

    Let $\Sigma$ be a complete fan in $\R^2$ such that two rays have opposite directions and thus make a straight line $L$ passing through the origin $O$. Pick a line $L'$ parallel to $L$ and denote by $r_1,r_2,\dotsc,r_k$ the rays intersecting $L'$ at the points $p_1,p_2,\dotsc,p_k$ respectively. Therefore $r_i = \overrightarrow{Op_i}$ for all $i$. One can obtain the points $p'_1,p'_2, \dotsc, p'_k$ by translating $p_1,p_2,\dotsc,p_k$ at the same time parallel to $L$. Then a new complete fan, called a \emph{shift along $L$}, is determined by replacing $r_i$ with $\overrightarrow{Op'_i}$. Due to Proposition~\ref{prop:shift}, two complete non-singular fans with $m$ rays in $\R^2$ are shifts of each other if they are connected by an edge in $D(P_m)$. The edge will have the color $i$ or $j$ when $L = r_i \cup r_j$. A notable fact is that, up to change of basis, one does not need to distinguish at which side of the line $L$ a shift has been performed. Let us explain this (for non-singular cases). Let $\Sigma$ be a complete non-singular fan in $\R^2$ given by the characteristic matrix
    \[
        \lambda = \begin{pmatrix}
            1  & 0 & \cdots &  x_i  & \cdots & -1 & \cdots & x_j & \cdots &  x_m \\
            0  & 1 & \cdots & y_i   & \cdots & 0  & \cdots & y_j & \cdots &  y_m
        \end{pmatrix}.
    \]
    If you perform a shift at the upper half-plane, then it will have the form
    \[
        \begin{pmatrix}
            1  &  e   & \cdots &  x_i+y_ie  & \cdots & -1 & \cdots & x_j & \cdots &  x_m \\
            0  &  1   & \cdots &  y_i       & \cdots & 0  & \cdots & y_j & \cdots &  y_m
        \end{pmatrix}
    \]
    and by a row operation, it becomes
    \[
        \lambda_e = \begin{pmatrix}
            1  &  0   & \cdots &  x_i       & \cdots & -1 & \cdots & x_j-y_je & \cdots &  x_m-y_me \\
            0  &  1   & \cdots &  y_i       & \cdots & 0  & \cdots & y_j & \cdots &  y_m
        \end{pmatrix},
    \]
    which is just a shift at the lower half-plane.

    Once we have found every edge of $D(P_m)$, the next step is to find every realizable square in $D(P_m)$. By an observation of above paragraph, an immediate candidate of a realizable square goes with a line $L$ such that $L = r_1 \cup r_\ell$ is the $x$-axis. Its standard form looks like
    \[
        \begin{pmatrix}
            1  &  0  &  \cdots  &  x_i  &  \cdots  &  -1  &  0  &  x_j    &  \cdots  &    x_m \\
            0  &  0  &  \cdots  &  y_i  &  \cdots  &   0  &  0  &  y_j    &  \cdots  &    y_m \\
            -1 &  1  &  \cdots  &    0  &  \cdots  &   0  &  0  &  -y_je  &  \cdots  &  -y_me \\
            0  &  0  &  \cdots  &    0  &  \cdots  &  -1  &  1  &  -y_jf  &  \cdots  &  -y_mf
        \end{pmatrix}.
    \]
    One observes that this matrix certainly represents a realizable square
    \begin{equation}\label{eq:square}
        \xymatrix{
                \lambda \ar@{-}[rr]^1 \ar@{-}[dd]^\ell  & & \lambda_{e} \ar@{-}[dd]^\ell \\
                  &   &  \\
                \lambda_{f}\ar@{-}[rr]^1 &   & \lambda_{e-f} }.
    \end{equation}
    \begin{figure}
        \begin{tikzpicture}[scale=2]
            \draw   [color=black] (0,0) -- (1,0)
                                  (0,0) -- (0,1)
                                  (0,0) -- (0,-1)
                                  (2.5,0) -- (3.5,0)
                                  (2.5,0) -- (2.5,1)
                                  (2.5,0) -- (2.5,-1);
            \draw   [color=red] (0,0) -- (-1,1)
                                (2.5,0) -- (1.5,0);
            \draw   [color=blue] (0,0) -- (-1,0)
                                 (2.5,0) -- (1.5,-1);
        \end{tikzpicture}
        \caption{Two fans given by $\lambda$ and $\lambda''$ before and after a shift along the $y$-axis when $f<0$. The red ray must exist since $\lambda''$ can be shifted along the $x$-axis.}\label{fig:shift}
    \end{figure}
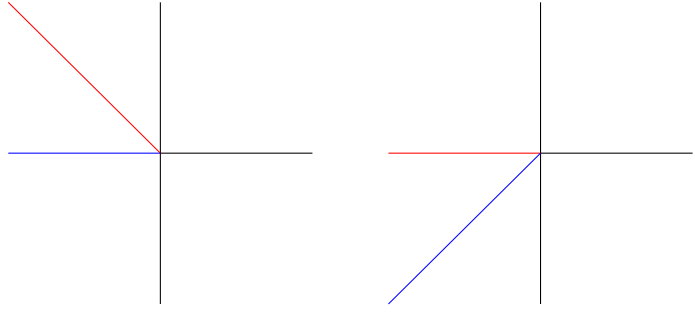
    \begin{proposition}\label{prop:square}
        Up to change of basis of $\Z^2$ and relabeling of vertices of $P_m$, every irreducible realizable squares in $D(P_m)$ has the form \eqref{eq:square} when $r_1$ and $r_\ell$ have opposite directions. Conversely, every square of the form \eqref{eq:square} is realizable whenever $r_1$ and $r_\ell$ have opposite directions. In particular, there is no irreducible realizable cube of dimension $\ge 3$ in $D(P_m)$.
    \end{proposition}
    \begin{proof}
        {First, let us consider the last assertion. Suppose that one has a realizable cube over $\wed_{i,j,k}K$ for $1\le i<j<k \le m$. Then two of $r_i, r_j, r_k$ are not parallel and by Proposition~\ref{prop:shift}, the cube is reducible.}

        If there is a realizable square other than the form \eqref{eq:square}, then it must involve two shifts along two different lines. Up to $\R$-basis change of $\R^2$, the two lines are the $x$- and $y$-axis respectively. Let us describe a ``standard form'' for the square
        
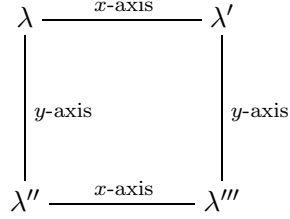
\begin{equation*}
            \xymatrix{
                    \lambda \ar@{-}[rr]^{\text{$x$-axis}} \ar@{-}[dd]^{\text{$y$-axis}}  & & \lambda' \ar@{-}[dd]^{\text{$y$-axis}} \\
                      &   &  \\
                    \lambda''\ar@{-}[rr]^{\text{$x$-axis}} &   & \lambda''' }
        \end{equation*}
        centered at $\lambda$.
        There are eight kinds of vectors of $\lambda$. The fan given by $\lambda$ must contain the rays generated by $\binom{1}{0},\binom{0}{1},\binom{-1}{0}, \binom{0}{-1}$. For simplicity of notation, we assume that the four vectors $\binom{x_i}{y_i},\binom{x_j}{y_j},\binom{x_k}{y_k},\binom{x_\ell}{y_\ell}$ are representatives of vectors of $\lambda$ in the open quadrant I, II, III, and IV respectively (you just pass it if a quadrant does not contain a ray). Therefore, the matrix for $\lambda$ is written as
        \[
            \lambda = \begin{pmatrix}
                1  &   x_i & 0  &  x_j  &   -1 & x_k   & 0  & x_\ell    \\
                0  &   y_i & 1  &  y_j  &   0  & y_k   & -1 & y_\ell
            \end{pmatrix}
        \]
        and the standard form for the square will look like
        \[
            \begin{pmatrix}
                1 & 0 &   x_i & 0 & 0 &  x_j  &   -1 & x_k   & 0  & x_\ell    \\
                0 & 0 &   y_i & 1 & 0 &  y_j  &   0  & y_k   & -1 & y_\ell    \\
                -1& 1 &   0   & 0 & 0 &   0   &   0  & -y_ke & e  & -y_\ell e \\
                0 & 0 &   0   &-1 & 1 & -x_jf &   f  & -x_kf & 0  &  0
            \end{pmatrix},
        \]
        where every variable is a real number, not necessarily an integer. One checks that
        \[
            \lambda' = \begin{pmatrix}
                1  &   x_i & 0  &  x_j  &   -1 & x_k-y_ke   & e  & x_\ell-y_\ell e    \\
                0  &   y_i & 1  &  y_j  &   0  & y_k   & -1 & y_\ell
            \end{pmatrix},
        \]
        \[
            \lambda'' = \begin{pmatrix}
                1  &   x_i & 0  &  x_j       &   -1 & x_k        & 0  & x_\ell    \\
                0  &   y_i & 1  &  y_j-x_jf  &   f  & y_k-x_kf   & -1 & y_\ell
            \end{pmatrix},
        \]
        and
        \[
            \lambda''' = \begin{pmatrix}
                1  &   x_i & 0  &  x_j       &   -1 & x_k-y_ke   & e  & x_\ell-y_\ell e    \\
                0  &   y_i & 1  &  y_j-x_jf  &   f  & y_k-x_kf   & -1 & y_\ell
            \end{pmatrix}
        \]
        by direct calculation for projections of characteristic maps (see \cite{CP13} and \cite{CP15}). Note that $x_k,y_k,x_\ell,y_\ell,$ and $e $ are all nonzero. Therefore, the two shifts along $x$-axis corresponding to the edges $\overline{\lambda\lambda'}$ and $\overline{\lambda''\lambda'''}$ move the same number of rays. But this is impossible since $\lambda$ and $\lambda''$ have different number of rays in the lower half-plane. See Figure~\ref{fig:shift}.
    \end{proof}

\section{The Projectivity: a proof of Theorem~\ref{thm:cpisproj}} \label{sec:3}
The goal of this section is to show the following:
\begin{theorem}[Theorem~\ref{thm:cpisproj}]
    Every toric manifold over $P_m(J)$ is projective for any $m\ge 3$ and an $m$-tuple $J \in \Z_+^m$.
\end{theorem}
The main tool used here is the Shephard diagram for wedges \cite{CP13} (refer \cite{She71}, \cite{Ewa86}, \cite{Ewa96} for details of Shephard diagrams). First, let us define the Shephard diagram of a complete simplicial fan.

\begin{lemma}\label{lem:sumzero}
    Let $\Sigma$ be a fan in $\R^n$ with $m$ rays. If $\Sigma$ is complete, then there is a set of nonzero vectors $\{u_1,\dotsc,u_m\}$ such that $u_i$ is on the $i$th ray for $i=1,\dotsc,m$ and $u_1+\dotsb+u_m = 0$.
\end{lemma}

\begin{proof}
    There are many possible proofs of this basic fact. One of them is as follows. Pick $m$ vectors $w_1,\dotsc,w_m$ so that $w_i$ generates $r_i$ for each $i$. Then we have
    \[
        w_1 + w_2 + \dotsb + w_m + (-w_1) + (-w_2) + \dotsb + (-w_m) = 0.
    \]
    Note that every vector $-w_i$ is in some cone in $\Sigma$ and thus each of them can be written as a linear combination of $w_1,\dotsc,w_m$ with nonnegative coefficients. Gathering all of them gives the wanted result.
    %In \cite{CP13}, for any positive reals $a_i$, $i=1,\dotsc,m$, observe that a linear transform of $(a_1x_1,\dotsc, a_mx_m)$ is $(a_1^{-1}\overline{x}_1,\dotsc,a_m^{-1}\overline{x}_m)$. A proof of the lemma is straightforward from Corollary~5.2 and Lemma~5.3 of \cite{CP13}, together with the above fact of the linear transform.
\end{proof}

Let $X = (u_1,\dotsc,u_m)$ be a sequence of vectors of the above lemma. The matrix
\[
    \begin{pmatrix}
        & & \\
        u_1 & \dotsb & u_m \\
        & & \\
    \end{pmatrix}_{n \times m}
\]
has rank $n$ and thus there exists a sequence of vectors $\{\widehat{u}_i\}_{1\le i \le m}\subset \R^{m-n-1}$ such that the set $\{(\widehat{u}_i,1)\mid 1\le i \le m\}$ is linearly independent and
\[
    \begin{pmatrix}
        & & \\
        u_1 & \dotsb & u_m \\
        & & \\
    \end{pmatrix}_{n \times m}
    \left(
        \begin{array}{ccc|c}
            & \widehat{u}_1 & & 1 \\
            & \vdots & & \vdots \\
            & \widehat{u}_m & & 1
        \end{array}
    \right)_{m\times (m-n)} = O.
\]%
%             $\widehat{\Delta}:=\{\widehat{x}_i\}$ is called a {Shephard diagram} of $\Delta$. This is defined up to basis change.
%
%
%
% Consider the space of linear dependencies of $X$
%\[
%        A_X := \{(\alpha_1,\dotsc,\alpha_m)\in \R^m \mid \alpha_1u_1+\dotsb+\alpha_mu_m = 0 \}
%\]
%which is an $(m-n)$-dimensional vector space. Reminding that $(1,\dotsc,1) \in A_X$, choose a basis of $A_X$ containing $(1,\dotsc,1)$ and write down them as rows of a matrix
%\[
%    \begin{pmatrix}
%        \alpha_{11} & & \cdots & & \alpha_{1m} \\
%        \vdots & & & & \vdots \\
%        \alpha_{m-n-1,1} & & \cdots & & \alpha_{m-n-1,m} \\
%        1  & & \cdots & & 1
%    \end{pmatrix}_{(m-n)\times m}.
%\]
%After deleting the row of ones, one obtains a sequence of column vectors
%\[
%    \begin{pmatrix}
%        \alpha_{11} & & \cdots & & \alpha_{1m} \\
%        \vdots & & & & \vdots \\
%        \alpha_{m-n-1,1} & & \cdots & & \alpha_{m-n-1,m}
%    \end{pmatrix}_{(m-n-1)\times m} =:
%    \begin{pmatrix}
%        \widehat{u}_1 & \cdots & \widehat{u}_m
%    \end{pmatrix}
%\]

\begin{definition}\label{def:shephard}
    The sequence $\widehat{\Sigma}:=\{\widehat{u}_i\}$ is called a {Shephard diagram} of the fan $\Sigma$.
\end{definition}

\begin{remark}
    The Shephard diagram is the inverse transform of the famous Gale transform. Beware that a Shephard diagram is not uniquely determined.
\end{remark}
\begin{remark}
    We frequently use the abbreviation $\widehat i = \widehat u_i$ when $u_i$ is a generator of the ray $r_i$ of $\Sigma$.
\end{remark}
\begin{example}(\cite{CP13})\label{exa:shephard}
    This example is presented to describe an explicit calculation of a Shephard diagram. Since every toric manifold over the pentagon is a blow-up of a Hirzebruch surface, its characteristic map is given by the following (up to basis change of $\Z^2$)
    \[
        \lambda_d=\begin{pmatrix}
            1   &   0   &   -1  &   -1  &   d\\
            0   &   1   &   1   &   0   &   -1
        \end{pmatrix}.
    \]
    Suppose that $d\ge 0$. To make the sum of column vectors zero, we multiply a suitable positive real number to each column of $\lambda_d$, resulting
    \[
        A=\begin{pmatrix}
            2   &   0   &   -1  &   -2d-1  &   2d\\
            0   &   1   &   1   &   0   &   -2
        \end{pmatrix}
    \]
    and we find a matrix $B$ of maximal rank which contains the column $(1,1,1,1,1)^T$ such that $AB=O$, for example
    \[
        \begin{pmatrix}
            2   &   0   &   -1  &   -2d-1  &   2d\\
            0   &   1   &   1   &   0   &   -2
        \end{pmatrix}
        \left(\begin{array}{cc|c}
            1 & -d & 1 \\
            -2& 2  & 1 \\
            2 & 0  & 1 \\
            0 & 0  & 1 \\
            0 & 1  & 1
        \end{array}\right) = O.
    \]
    Therefore we obtain $\widehat{\Sigma} = \{\widehat{1},\widehat{2},\widehat{3},\widehat{4}, \widehat{5} \} = \{(1,-d),(-2,2),(2,0),(0,0),(0,1) \}$ when we abbreviate $\widehat{i} = \widehat{u}_i$.
\end{example}

Before introducing Shephard's criterion for projectivity, we need to explain some definitions and notions. A \emph{relative interior} of a subset $A$ in $\R^n$, denoted by $\relint A$, is defined to be the interior of $A$ in the minimal affine space containing $A$.
\begin{definition}
A simplicial fan $\Sigma$ is called \emph{strongly polytopal} if there is a simplicial polytope $P$ such that $0\in\relint P$ and each cone of $\Sigma$ is spanned by a proper face of $P$ and vice versa.
\end{definition}
It is a well-known fact that a toric variety is projective if and only if its corresponding fan is strongly polytopal. A non-singular projective toric variety is also called a \emph{projective toric manifold}.

Let $\widehat{\Sigma}$ be a Shephard diagram of a fan and $C\in \Sigma$ a cone which is a face of $\Sigma$. A \emph{coface} $\widehat{C}$ of $\Sigma$ is defined to be
\[
    \widehat{C} := \relint \conv \{\widehat{u}_i \mid u_i\text{ does not generate a ray of }C\}_{1\le i \le m}.
\]
%A subset $\sigma \subseteq [m]$ is called a \emph{coface} of $\Sigma$ if the cone spanned by the set $\{u_i \mid i\in [m]\setminus \sigma\}$ is a face of $\Sigma$.
The following theorem is a key to check whether a toric manifold is projective or not.

\begin{theorem}[Shephard's criterion]\cite{She71,Ewa86}\label{thm:shephard}
    A complete simplicial fan $\Sigma$ is strongly polytopal if and only if
    \[
        S(\widehat{\Sigma}) := \bigcap_{C\in\Sigma}\widehat{C} \ne\varnothing.
    \]
\end{theorem}

\begin{convention}
    When there is no danger of confusion, we denote an open convex polytope by its vertices. For example,
    \[
        \relint \conv \{\widehat u_1,\widehat u_2, \widehat u_3\} = \widehat u_1\widehat u_2 \widehat u_3 = \widehat 1 \widehat 2 \widehat 3.
    \]
\end{convention}
\begin{example}
    Let $\widehat{\Sigma}$ be the Shephard diagram of Example~\ref{exa:shephard}. Using the above theorem, we can check whether $\Sigma$ is strongly polytopal or not. %For simplicity of notation, let us denote a coface by its vertices. For example, if $C$ is generated by the rays $1$ and $2$, then $\widehat{C} = \relint \conv \{\widehat{3},\widehat{4},\widehat{5}\} = \widehat{3}\widehat{4}\widehat{5}$.
    Note that, to compute $S(\widehat{\Sigma})$, one needs to consider only cofaces of maximal cones. Therefore, $S(\widehat{\Sigma}) = \widehat{3}\widehat{4}\widehat{5} \cap \widehat{4}\widehat{5}\widehat{1} \cap \widehat{5}\widehat{1}\widehat{2} \cap \widehat{1}\widehat{2}\widehat{3} \cap \widehat{2}\widehat{3}\widehat{4}$ and it is the colored region of Figure~\ref{fig:shephard}. This is nonempty for every $d\ge 0$ and we conclude that $\Sigma$ is strongly polytopal for all $d\ge 0$.
    \begin{figure}
        \begin{tikzpicture}[scale = 1]
            \draw [help lines] (-2.5,-2.5) grid (2.5,2.5);
            \draw [thin, ->] (-2.5,0) -- (2.5,0);
            \draw [thin, ->] (0,-2.5) -- (0,2.5);
            \fill [red] (0,0) -- (0.33,0) -- (0,1);
            \coordinate [label = right:{$\widehat{1}$}] (1) at (1,-2);
            \coordinate [label = left :{$\widehat{2}$}] (2) at (-2,2);
            \coordinate [label = above:{$\widehat{3}$}] (3) at (2,0);
            \coordinate [label = left :{$\widehat{4}$}] (4) at (0,0);
            \coordinate [label = above right:{$\widehat{5}$}] (5) at (0,1);
            \draw [ultra thick, miter limit=2] (1)--(2)--(3)--(4)--(5)--cycle;
        \end{tikzpicture}
        \caption{A Shephard diagram of $\Sigma$ when $d=2$.}\label{fig:shephard}
    \end{figure}
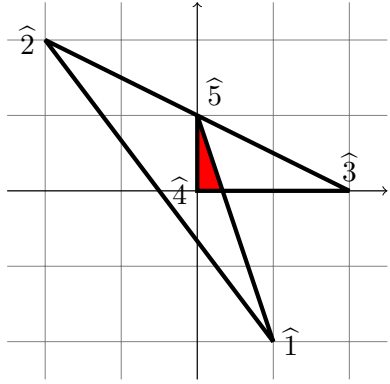
\end{example}

\begin{proposition}\cite[Proposition~5.9]{CP13}\label{prop:shephardofwedge}
    Let $K$ be a star-shaped simplicial complex with vertex set $[m]$, $\Sigma$ a fan over $\wed_1 K$, and $\Sigma_1, \Sigma_2$ the fans over $K$ obtained by projections from $\Sigma$. If $$\widehat{\Sigma} = \{\widehat{1}_1, \widehat{1}_2, \widehat{2}, \dotsc, \widehat{m}\}$$ is a Shephard diagram of $\Sigma$, then $$ \widehat{\Sigma}_1 = \{\widehat{1}_2, \widehat{2}, \dotsc, \widehat{m} \}$$ and $$ \widehat{\Sigma}_2 = \{\widehat{1}_1, \widehat{2}, \dotsc, \widehat{m} \}$$ are Shephard diagrams of $\Sigma_1$ and $\Sigma_2$. Moreover,
    \[
        S(\widehat{\Sigma}) = S(\widehat{\Sigma}_1) \cap S(\widehat{\Sigma}_2).
    \]
\end{proposition}
This fact can certainly be generalized for fans over $K(J)$ in an obvious way.

Let us given a complete non-singular fan $\Sigma$ over $P_m(J)$ when $J = (j_1,\dotsc,j_m)$. For the purpose to show projectivity, we can assume that the corresponding puzzle is irreducible. If not, it is a canonical extension and we can consider its projection without changing its Shephard diagram. By Proposition~\ref{prop:square}, we can assume that $r_1$ and $r_\ell$ are of opposite directions and $j_i = 1$ for $i\ne 1$ and $i \ne \ell$. In this setting, we obtain a Shephard diagram
\[
    \widehat{\Sigma} = \{\widehat{1}_1,\dotsc,\widehat{1}_{j_1},\widehat{\ell}_1,\dotsc, \widehat{\ell}_{j_\ell}\} \cup \{\widehat{\alpha} \mid 1< \alpha <\ell\} \cup \{\widehat{\beta} \mid \ell < \beta \le m\}\subset\R^{m-3}.
\]
We put $A := \{ {\alpha} \mid 1< \alpha <\ell\}$ and $B:=\{ {\beta} \mid \ell < \beta \le m\}$. First, we consider the case $j_1= j_\ell = 1$. Then $\Sigma$ is just a fan over the plane. In this case we write $1_1 = 1$ and $\ell_1 = \ell$.

\begin{lemma}\label{lem:cofaceisasimplex}
    If $C$ is a maximal cone of $\Sigma$ in $\R^2$, then the coface $\widehat{C}\subset \R^{m-3}$ is an open $(m-3)$-simplex.
\end{lemma}
\begin{proof}
    This is a direct consequence of Lemma~II.4.10. of \cite{Ewa96}.
\end{proof}

\begin{lemma}\label{lem:radon}
    In $\R^{m-3}$, the set $\{\widehat\alpha\}_{\alpha\in A} \cup \{\widehat\beta\}_{\beta\in B}$ affinely spans an $(m-4)$-dimensional hyperplane $H$. The sets $\relint \conv \{\widehat{\alpha}\}$ and $\relint\conv\{\widehat\beta\} $ intersect at exactly one point, say $R\in \R^{m-3}$.
\end{lemma}
\begin{proof}
    The point $u_i = \binom{x_i}{y_i}$ is in the upper half-plane if $i\in A$ and in the lower half-plane if $i \in B$. Therefore $y_\alpha >0 $ for $\alpha \in A$ and $y_\beta <0 $ for $\beta \in B$. Noticing that the notation $\widehat{i} = \widehat{u}_{i}$, observe that we have the relation
    \begin{equation}\label{eq:affinerelation}
        \sum_{\alpha\in A}y_\alpha \widehat\alpha + \sum_{\beta \in B}y_\beta \widehat\beta = 0
    \end{equation}
    by the very definition of Shephard diagrams. This is actually an affine relation since $\sum y_\alpha + \sum y_\beta = 0$. Therefore $\{\widehat\alpha\} \cup \{\widehat\beta\}\in \R^{m-3} $ affinely spans a proper subspace $H$ of $\R^{m-3}$. It is easy to show the dimension of $H$ is $m-4$ by applying Lemma~\ref{lem:cofaceisasimplex} for a maximal coface whose vertices are in $H$ except $\widehat{1}$. Finally, the construction of the point $R$ is again by the above affine relation. Put $s = \sum_{\alpha \in A}{y_\alpha} = -\sum_{\beta\in B}y_\beta$. Then
    \begin{equation}\label{eq:radonpoint}
        R = \frac1s  \sum_{\alpha\in A}y_\alpha \widehat\alpha = \frac1s \sum_{\beta \in B}y_\beta \widehat\beta
    \end{equation}
    works for the proof.
\end{proof}
\begin{remark}
    Lemma~\ref{lem:radon} looks similar to the classical Radon's theorem. The point $R$ is the Radon point.
\end{remark}

\begin{lemma}\label{lem:sameside}
    The points $\widehat{1}$ and $\widehat{\ell}$ lie on the same side of $H$.
\end{lemma}
\begin{proof}
    By the projectivity of toric manifolds of complex dimension $2$, the cofaces $\widehat{C}$ intersects to each other. By Lemma~\ref{lem:cofaceisasimplex}, $\widehat{1}\notin H$ and $\widehat{\ell}\notin H$. If $\widehat{1}$ and $\widehat{\ell}$ are on the different sides of $H$, a coface with vertex $\widehat{1}$ and another coface with vertex $\widehat{\ell}$ do not intersect.
\end{proof}

\begin{proposition}\label{prop:smalltriangle}
    Every coface of $\Sigma$ contains $B_\epsilon(R) \cap \widehat{1}\widehat{\ell}R$ for some $\epsilon > 0$.
\end{proposition}
\begin{proof}
    Let $C$ be a maximal cone of $\Sigma$ in $\R^2$. Then one of the following is true:
    \begin{enumerate}
        \item $r_1 \subset C$ and $\widehat\ell$ is a vertex of $\widehat C$.
        \item $r_\ell \subset C$ and $\widehat 1 $ is a vertex of $ \widehat C$.
        \item $r_1,r_\ell \not\subset C$ and $\widehat 1,\widehat\ell $ are vertices of $\widehat C$.
    \end{enumerate}
    For every coface $\widehat C$, note that the vertex set of $\widehat C$ either contains $\widehat\alpha$ for all $\alpha\in A$ or contains $\widehat\beta$ for all $\beta\in B$. Therefore $\partial\widehat C$ contains $R$ and Case (3) is easily dealt with. For Case (1) and (2), we can safely assume that $C$ is the one spanned by $r_1$ and $r_\ell$. We perform a calculation similar to that of Lemma~\ref{lem:halfplane}. For the notation of the matrix, we divide the set of the rays of $\Sigma$ into five categories: $\{1\}$,  $A$, $\{\ell\}$, $B\setminus\{m\}$, and $\{m\}$ (the last ray) and the relation for the Shephard transform looks like the following.
    \begin{equation*}
        \begin{pmatrix}
            x_1 & x_\alpha & x_\ell & x_\beta & x_m \\
            0   & y_\alpha & 0      & y_\beta & y_m
        \end{pmatrix}
        \left(\begin{array}{c|c}
            \widehat{1} & 1 \\
            \widehat{\alpha} & 1\\
            \widehat{\ell} & 1 \\
            \widehat{\beta} & 1 \\
            \widehat{m}   & 1
        \end{array}\right) = O.
    \end{equation*}
    From the above equation, we have two affine relations
    \[
        x_1\widehat 1 + \sum x_\alpha \widehat\alpha + x_\ell\widehat\ell + \sum_{\beta\ne m} x_\beta\widehat\beta + x_m \widehat m = 0
    \]
    and
    \[
        \sum y_\alpha\widehat \alpha + \sum_{\beta\ne m} y_\beta\widehat\beta + y_m \widehat m = 0.
    \]
    If we eliminate the term $\widehat m$ from the above two relations, we obtain
    \[
        x_1y_m \widehat 1 + \sum_\alpha(x_\alpha y_m - x_m y_\alpha) \widehat \alpha + x_\ell y_m \widehat \ell + \sum_{\beta\ne m} (x_\beta y_m - x_m y_\beta)\widehat\beta = 0
    \]
    which is also an affine relation. From the facts $x_\ell y_m >0$ and $x_\beta y_m - x_m y_\beta >0$, one concludes that the affine hull of $\{\widehat\alpha,\widehat 1\}$ intersects with $\relint \conv \{\widehat\beta, \ell\mid  {\beta\ne m}\}$. Furthermore, because  $x_1 y_m <0$, any ray from a point of $\relint \conv \{\alpha\}$ to $\widehat 1$ intersects with $\relint \conv \{\widehat\beta, \ell\mid {\beta\ne m}\}$. Now it is easy to show that $\widehat C = \relint \conv \{\widehat\alpha, \widehat\beta, \widehat\ell \mid \beta\ne m \}$ contains $B_\epsilon(R) \cap \widehat{1}\widehat{\ell}R$ for some $\epsilon > 0$.
    %Then $\widehat C = \widehat 3 \widehat 4 \cdots \widehat \ell \cdots \widehat m$. Consider a chain of cones
%    \[
%        \conv(r_1\cup r_2),\,\conv(r_2\cup r_3),\,\dotsc,\conv(r_{\ell-2},r_{\ell-1})
%    \]
%    and their corresponding cofaces. By observing

\end{proof}

Next, we consider the case $j_1=2$ and $j_\ell = 1$. Let us write $1_1=1,\,1_2=1'$, and $\ell_1 = \ell$. %The following lemma is a key fact to prove projectivity.

\begin{lemma}\label{lem:halfplane}
    The points $\widehat{1}$, $\widehat{1}'$, $\widehat{\ell}$, and the Radon point $R$ are on 2-dimensional affine space.
\end{lemma}
\begin{proof}
    For simplicity of notation, we divide the set of the rays of $\Sigma$ into four parts: $\{1,1'\}$,  $A$, $\{\ell\}$, and $B$. By Proposition~\ref{prop:shephardofwedge}, we have $ \widehat{\Sigma}_2 = \{\widehat{1}_1, \widehat{2}, \dotsc, \widehat{m} \}$. A Shephard diagram for $\Sigma_2$ is given by the identity
    \begin{equation*}
        \begin{pmatrix}
            x_1 & x_\alpha & x_\ell & x_\beta \\
            0   & y_\alpha & 0      & y_\beta
        \end{pmatrix}
        \left(\begin{array}{c|c}
            \widehat{1} & 1 \\
            \widehat{\alpha} & 1\\
            \widehat{\ell} & 1 \\
            \widehat{\beta} & 1
        \end{array}\right) = O.
    \end{equation*}
    Here, we get the following
    \begin{equation}\label{eq:sigma2}
        x_1\widehat{1} + \sum x_\alpha \widehat\alpha + x_\ell\widehat\ell + \sum x_\beta\widehat\beta =  0
    \end{equation}
    and
    \begin{equation}\label{eq:sum2}
        x_1 + \sum x_\alpha + x_\ell + \sum x_\beta = 0.
    \end{equation}
    The rays of $\Sigma_1$ are generated by the column vectors of the following matrix
    \[
        \begin{pmatrix}
            x_1 & x_\alpha & x_\ell & x_\beta-ey_\beta \\
            0   & y_\alpha & 0      & y_\beta
        \end{pmatrix}
    \]
    due to Proposition~\ref{prop:shift}. After multiplying a positive real $a_i$ to each column (Lemma~\ref{lem:sumzero}), a Shephard diagram for $\Sigma_1$ is given by the identity
    \[
        \begin{pmatrix}
            a_1x_1 & a_\alpha x_\alpha & a_\ell x_\ell & a_\beta(x_\beta-ey_\beta) \\
            0   & a_\alpha y_\alpha & 0      & a_\beta y_\beta
        \end{pmatrix}
        \left(\begin{array}{c|c}
            \widehat{1}' & 1 \\
            \widehat{\alpha} & 1\\
            \widehat{\ell} & 1 \\
            \widehat{\beta} & 1
        \end{array}\right) = O.
    \]
    This identity gives another affine relation $\sum a_\alpha y_\alpha \widehat{\alpha} + \sum a_\beta y_\beta \widehat{\beta} = 0$. But up to scaling, \eqref{eq:affinerelation} is the unique affine relation between $\widehat\alpha$'s and $\widehat\beta$'s. Therefore $a_\alpha=a_\beta$ for all $\alpha\in A$ and $\beta\in B$. By dividing the leftmost matrix by $a_\alpha$, we can further assume that $a_\alpha=a_\beta = 1$. Hence the result is
    \[
         \begin{pmatrix}
            a_1x_1 &  x_\alpha & a_\ell x_\ell &  x_\beta-ey_\beta  \\
            0   &  y_\alpha & 0      &  y_\beta
        \end{pmatrix}
        \left(\begin{array}{c|c}
            \widehat{1}' & 1 \\
            \widehat{\alpha} & 1\\
            \widehat{\ell} & 1 \\
            \widehat{\beta} & 1
        \end{array}\right) = O,
    \]
    which deduces
    \begin{equation}\label{eq:sigma1}
        a_1x_1\widehat{1}' + \sum x_\alpha\widehat\alpha + a_\ell x_{\ell} \widehat\ell + \sum(x_\beta-ey_\beta)\widehat\beta = 0
    \end{equation}
    and
    \begin{equation}\label{eq:sum1}
        a_1x_1 + \sum x_\alpha + a_\ell x_\ell + \sum x_\beta -e\sum y_\beta = 0.
    \end{equation}
    By subtracting \eqref{eq:sigma2} from \eqref{eq:sigma1}, we get
    \[
        -x_1\widehat{1} + a_1x_1\widehat{1}' +  (a_\ell x_{\ell}-x_\ell) \widehat\ell - e\sum y_\beta \widehat\beta   = 0.
    \]
    Remember \eqref{eq:radonpoint} and then
    \begin{equation}\label{eq:1ellplane}
        -x_1\widehat{1} + a_1x_1\widehat{1}' +  (a_\ell x_{\ell}-x_\ell) \widehat\ell + esR   = 0.
    \end{equation}
    Similarly subtracting \eqref{eq:sum2} from \eqref{eq:sum1}, we get
    \begin{equation}\label{eq:1ellplanesumzero}
        -x_1 + a_1x_1  + a_\ell x_\ell -  x_\ell + es   = 0.
    \end{equation}
    The identities \eqref{eq:1ellplane} and \eqref{eq:1ellplanesumzero} provide the wanted affine relation.
\end{proof}

{In summary, the following holds.}

\begin{proposition}\label{prop:halfplane}
    Let $\Sigma$ be a complete non-singular fan over $P_m(J)$. Assume that $r_1$ and $r_\ell$ are of opposite directions and $j_i = 1$ for $i\ne 1$ and $i \ne \ell$. Then all of the points $\widehat{1}_i$ and $\widehat{\ell}_k$ lie on an open half-space $\mathcal{H} $ of dimension 2 transversally intersecting $H$ such that $R \in \partial\mathcal{H} \subseteq H$. Furthermore, no ray $\overrightarrow{R\widehat{\ell}_k}$ is between $\overrightarrow{R\widehat{1}_i}$ and $\overrightarrow{R\widehat{1}_j}$ and vice versa.
\end{proposition}
\begin{proof}
    By Lemma~\ref{lem:halfplane}, every point $\widehat{1}_i$ and $\widehat{\ell}_k$ lie on an affine 2-space. One also reminds Lemma~\ref{lem:sameside}. For the last assertion, see the affine relation \eqref{eq:1ellplane} where the coefficients of $\widehat{1}$ and $\widehat{1}'$ have different signs.
\end{proof}
\begin{figure}
    \begin{tikzpicture}[scale = 1.2]
        \coordinate [label = below:{$R$}] (r) at (0,0);
        \coordinate [label = above:{$\widehat{1}_1$}] (1) at (-2,2.5);
        \coordinate [label = above:{$\widehat{1}_j$}] (2) at (-.4,3);
        \coordinate [label = above:{$\widehat{\ell}_p$}] (3) at (1.6,3);
        \coordinate [label = above right:{$\widehat{\ell}_1$}] (4) at (2.6,2);
        \fill [red] (r) -- (-.1,.75) -- (.4,.75);
        \node [right] at (3,1) {$H$};
        \node [below right] at (1,0) {$\partial\mathcal{H}$};
        \node at (-.8,1.8) {$\cdots$};
        \node at (1.5,1.8) {$\cdots$};
        \draw [very thin] (-2,1) -- (3,1) -- (2,-1) -- (-3,-1) -- cycle;
        \draw [thick] (-1.5,0) -- (1.5,0);
        \draw [thick] (r)--(1) (r)--(2) (r)--(3) (r)--(4);
        %\draw [fill] (r) circle [radius=.1];
    \end{tikzpicture}
    \caption{The rays in the half-plane $\mathcal{H}$.}\label{fig:halfplane}
\end{figure}
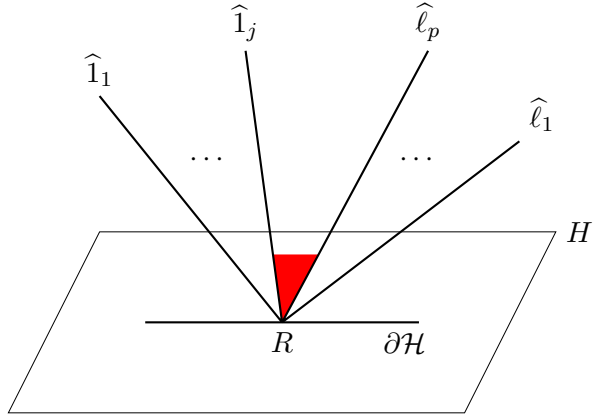

Now we are ready to prove the main theorem.

\begin{proof}[proof of Theorem~\ref{thm:cpisproj}]
    By Proposition~\ref{prop:halfplane}, every triangle $\widehat{1}_i\widehat{\ell}_kR$ contains the colored region in Figure~\ref{fig:halfplane} which is the intersection of $\widehat{1}_j\widehat{\ell}_pR$ and a neighborhood of $R$, where $\widehat{1}_j\widehat{\ell}_pR$ is the ``innermost'' triangle.
    Hence, every coface of $\Sigma$ contains $B_\epsilon(R) \cap \widehat{1}_j\widehat{\ell}_pR$ by Proposition~\ref{prop:smalltriangle} and $S(\Sigma) \supseteq B_\epsilon(R) \cap \widehat{1}_j\widehat{\ell}_pR$. In particular, $S(\Sigma)\ne\varnothing$.
\end{proof}

\end{document}